\documentclass{amsart}
\usepackage{amsmath, amsthm, amssymb, graphicx, tikz, zref}
\usetikzlibrary{arrows}
\usepackage[justification=centering]{caption}
\usepackage{enumitem}
\usepackage{cite}
\newtheorem{theorem}{Theorem}[section]

\newtheorem{lemma}[theorem]{Lemma}
\newtheorem{example}[theorem]{Example}

\newtheorem{proposition}[theorem]{Proposition}

\newtheorem{definition}[theorem]{Definition}

\newcommand{\trop}{{\mathcal{H}}}
\newcommand{\relint}{\textnormal{Relint}}

\newcommand{\subdiv}[1]{{SD}_{#1}}
\newcommand{\tropical}[1]{{Trop}_{#1}}

\newcommand{\convh}{\textnormal{convh}}

\newcommand{\Hom}{\textnormal{Hom}}

\usepackage{hyperref}
\usepackage{xcolor}
\hypersetup{
    colorlinks,
    linkcolor={red!50!black},
    citecolor={blue!50!black},
    urlcolor={blue!80!black}
}

\title{Painted Tropical Complexes}
\author{Gabriel Kerr}
\author{Sophia Palcic}
\date{July 2023}

\begin{document}

\maketitle
\begin{abstract}
    We define the notion of a painted tropical $A$-complex and describe a poset structure on the set of all such complexes. We show that this poset is equivalent to the face lattice of a secondary polytope $\Sigma (\bar{A}_\alpha )$ where $\bar{A}_\alpha$ is built from $A$ and an additional point $\alpha$. As a central application, we show that multiplihedra are also secondary polytopes.
\end{abstract}
\section{Introduction}
In \cite{forcey}, Forcey gave a concrete realization of multiplihedra as polytopes which was further explored in \cite{mauwood}. Such multiplihedra give the combinatorial data needed to describe $A_\infty$-morphisms between $A_\infty$-algebras and have received renewed attention in \cite{bottman, lapmaz}. This paper presents an alternative construction which realizes multiplihedra as secondary polytopes of certain three dimensional polytopes. Rather than focus from the start on obtaining this specific result, our approach is to give a more general construction of a painting fan $\mathcal{F}_{A, \alpha}$ associated to a finite set $A$ and an additional marked $\alpha$, both in a real vector space. Our main result is Theorem~\ref{thm:mainthm} which states that the face lattice of $\mathcal{F}_{A, \alpha}$ is dual to that of a particular secondary polytope which we call the painting polytope of $A$ relative to $\alpha$. The idea being that multiplihedra are but one instance of this more general class of polytopes whose face lattice keeps track of linear paintings on a tropical hypersurface with fixed marked Newton polytope.

\section*{Acknowledgements:}
The first author would like to thank Nate Bottman for early discussions on multiplihedra and $2$-associahedra which formed the original inspiration for this paper. Both authors appreciate the input and suggestions given by Mikhail Mazin and Dasha Poliakova regarding this work. 

\section{Background and notation}
This section contains a brief account of constructions and results that can be found in the literature \cite{gross, gkz, bilstu}, but slightly recast to suit our purposes. A \textbf{marked polytope} is a pair $(Q, A)$ consisting of a convex polytope $Q$ in an $n$-dimensional real vector space $V$ along with a finite set $A$ in $Q$ whose convex hull is $Q$. We will call $(Q, A)$ a marked simplex if $\dim (Q) = |A| - 1$ (implying it is a simplex and $A$ consists only of its vertices). We denote by $\mathbb{R}^A$ the vector space of functions $\eta : A \to \mathbb{R}$ and write $\{e_a\}_{a \in A}$ for the standard basis. Fixing a particular $\eta \in \mathbb{R}^A$ induces two  polyhedral complexes, one subdividing the polytope $Q$ and the other subdividing the vector dual $V^* = \Hom_{\mathbb{R}} (V, \mathbb{R})$. The face lattices of these decompositions are dual to each other. For what follows we will refer to two running examples to illustrate the various structures which we introduce.
\begin{figure}[ht]
\begin{tikzpicture}[cross line/.style={preaction={draw=white, -, line width=6pt}}]
	\node[anchor=south west,inner sep=0] (image) at (3,0) {\includegraphics[scale=.3]{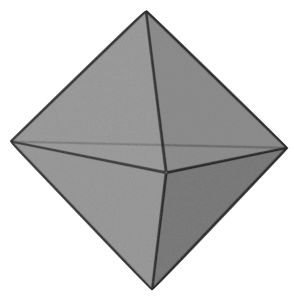}};
     \node[anchor=south west,inner sep=0] (image) at (-3,0) {\includegraphics[scale=.085]{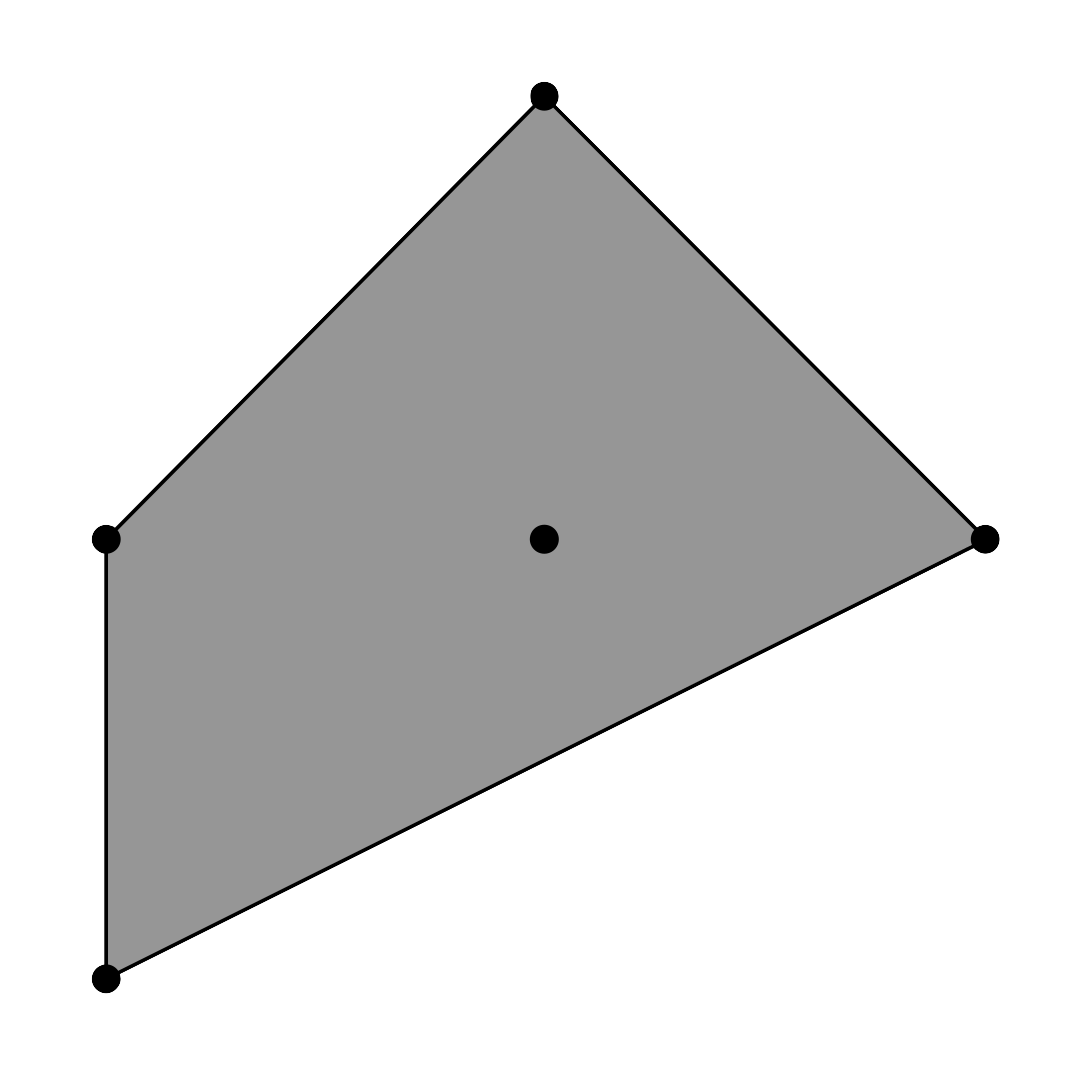}};
\end{tikzpicture}
\caption{\label{fig:example1a2a} Examples $(Q, A)$ and $(\tilde{Q}, \tilde{A})$.}
\end{figure}
\begin{example} \label{ex:qwi0}
The case of vertices of a polygon will be considered in Section~\ref{sec:multiplihedra}. In general, the point configurations will have interior markings or points on the boundary. For example, take
\begin{align*}
A = \left\{ (0,0), (1,0), (0,1), (-1,0), (-1,-1) \right\}
\end{align*}
and observe that $Q$ is a quadrilateral with the origin as an interior point as illustrated in Figure~\ref{fig:example1a2a}.
\end{example}
\begin{example} \label{ex:bp0}
For a $3$-dimensional example, we choose 
\begin{align*}
\tilde{A} = \left\{ (1, 0, 0), (0, 1, 0), (-1,-1, 0), (0, 0, 1), (0, 0, -1) \right\}.
\end{align*}
Sets such as this whose affine lattice of relations is rank $1$ are known as  circuits. Its convex hull $\tilde{Q}$ is a bipyramid and illustrated in Figure~\ref{fig:example1a2a}.
\end{example}

\subsection{Secondary polytopes} Given $\eta \in \mathbb{R}^A$, one constructs the polyhedron $Q_\eta$ as
\begin{align*}
    Q_\eta = \convh \left\{ ( a , r ) : r \leq - \eta (a) \right\} \subset V \times \mathbb{R}.
\end{align*}
Notice that projecting to $V$ simply takes $Q_\eta$ to $Q$. Taking the compact faces of $Q_\eta$ which form the upper boundary, and projecting thus results in a polyhedral subdivision $\mathcal{S}_\eta$ of $Q$.

Let us make this description more precise. Define the function $g_\eta : Q \to \mathbb{R}$ by taking
\begin{align*}
    g_\eta (v) := \max \{ r : (v, r) \in Q_\eta \}.
\end{align*}
Then the graph of $g_\eta$ is the upper boundary of $Q_\eta$ and it follows that $g_\eta$ is a concave piecewise linear function. Thus for each compact facet $F_i$ of $Q_\eta$, there is an affine function $g_{\eta, i} - c_{\eta, i} : V \to \mathbb{R}$ uniquely defined by
\begin{align} \label{eq:defgi}
    g_{\eta, i} (v) - c_{\eta, i} = g_\eta (v)
\end{align}
for $v$ in the projection $Q_i$ of $F_i$. Here $g_{\eta, i} \in V^*$ and $c_{\eta, i} \in \mathbb{R}$. Then for any $v \in Q$ we have
\begin{align*}
     g_{\eta, i} (v) \geq g_{\eta} (v)
\end{align*}
with equality if and only if $v$ is in $F_i$. In other words, the graph of $g_{\eta, i} - c_{\eta, i}$ is the hyperplane supporting the facet $F_i$. We may promote $Q_i$ to a marked polytope $(Q_i, A_i)$ by taking $A_i$ to be the elements in $A$ for which $g_{\eta, i} (a) - c_{\eta, i} = \eta (a)$. 
Then we call the collection of marked polytopes and their faces 
\begin{align*}
    \mathcal{S}_\eta = \{ (Q_i, A_i) \}
\end{align*} 
the subdivision of $Q$ induced by $\eta$ and write $\mathcal{S}_\eta (k)$ for the subcollection of $k$-dimensional faces of $\mathcal{S}_\eta$. 
\begin{figure}[ht]
\begin{tikzpicture}[cross line/.style={preaction={draw=white, -, line width=6pt}}]
	\node[anchor=south west,inner sep=0] (image) at (-3,0) {\includegraphics[scale=.28]{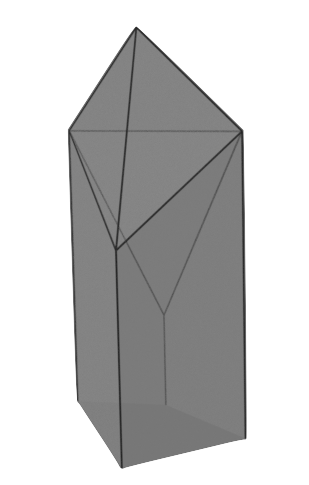}};
     \node[anchor=south west,inner sep=0] (image) at (3,1) {\includegraphics[scale=.085]{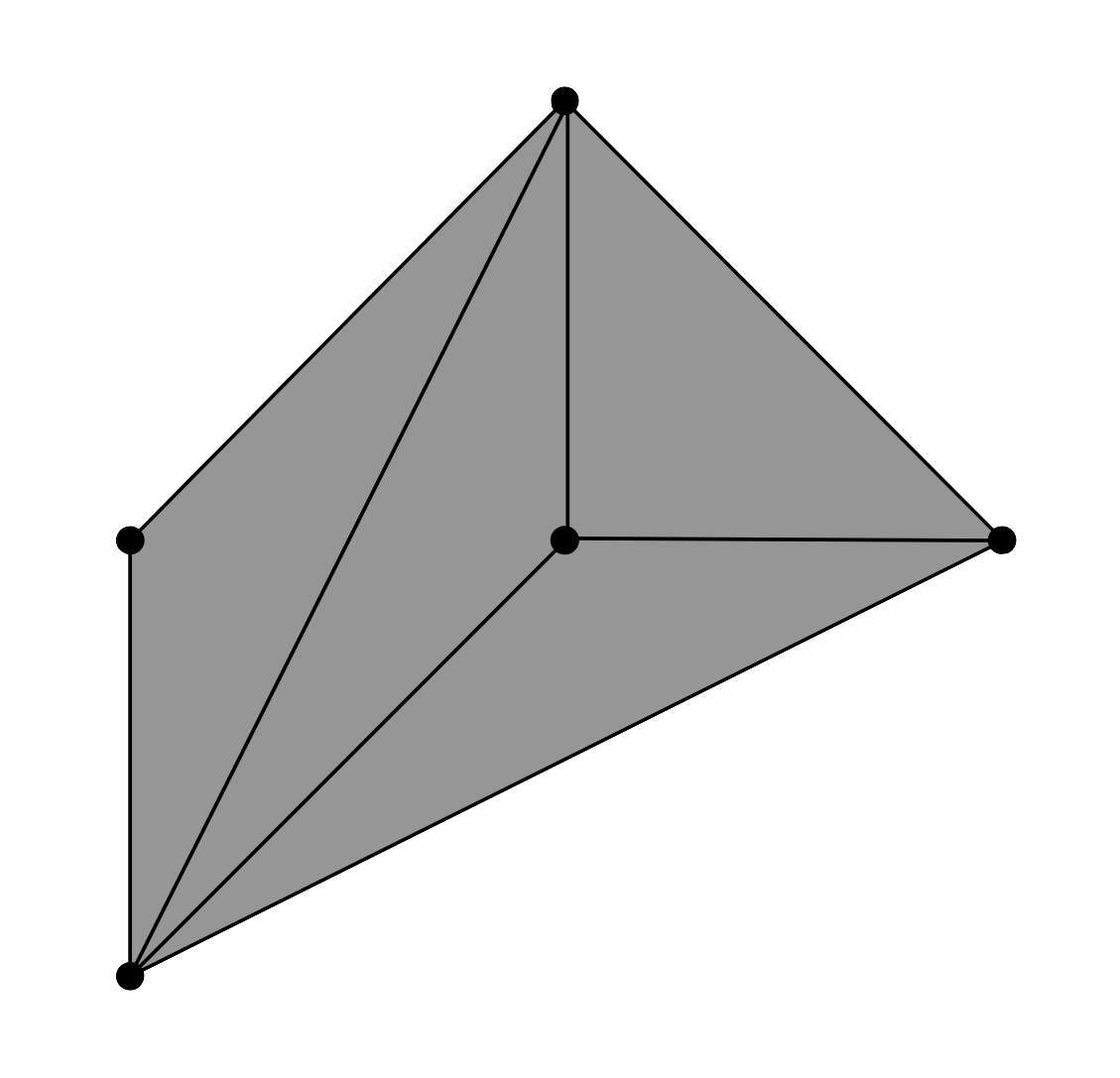}};
\end{tikzpicture}
\caption{\label{fig:example1b1c} The polyhedron $Q_\eta$ and its induced subdivision $\mathcal{S}_\eta$.}
\end{figure}

\begin{example} \label{ex:qwi1}
Figure~\ref{fig:example1b1c} illustrates the polyhedron $Q_\eta$ for 
\[\eta = - e_{(0,0)} + e_{(1,0)} + 2 e_{(-1,0)} \]
as well as the associate subdivision $\mathcal{S}_\eta$.
\end{example}

\begin{definition}\cite{gkz} 
A \textbf{subdivision} $\mathcal{S} = \{(Q_i, A_i)\}_{i \in I}$ of $(Q,A)$ is a collection of marked polytopes $(Q_i, A_i)$ with $A_i \subset A$ for which
\begin{enumerate}
    \item every face of $(Q_i, A_i)$ is in $\mathcal{S}$,
    \item the intersection of two polytopes is a face of each,
    \item the union of $Q_i$ is $Q$.
\end{enumerate}
The subdivision $\mathcal{S}$ is called \textbf{coherent} if it is induced by some $\eta \in \mathbb{R}^A$. If every $(Q_i, A_i)$ is a marked simplex, $\mathcal{S}$ is called a \textbf{triangulation}. 
\end{definition}

If $\mathcal{S}_1$ and $\mathcal{S}_2$ are subdivisions so that every $(Q_i, A_i)$ of $\mathcal{S}_2$ is subdivided by marked polytopes of $\mathcal{S}_1$, we say that $\mathcal{S}_1$ refines $\mathcal{S}_2$ and write $\mathcal{S}_1 \preceq \mathcal{S}_2$. This makes coherent subdivisions of $(Q,A)$ a poset.

We may stratify $\mathbb{R}^A$ by identifying two functions if and only if they induce identical subdivisions. We give a little notation for this by writing, for any coherent subdivision $\mathcal{S}$, 
\begin{align*}
    C^\circ_{\mathcal{S}} := \{ \eta \in \mathbb{R}^A : \mathcal{S}_\eta = \mathcal{S} \}
\end{align*}
and $C_{\mathcal{S}}$ for its closure. 

\begin{proposition}\cite{gkz} \label{prop:secondaryfan}
For any coherent subdivision $\mathcal{S}$, $C_{\mathcal{S}}$ is a polyhedral cone in $\mathbb{R}^A$. The collection 
\begin{align*}
    \mathcal{F}_{Sec A} = \{C_{\mathcal{S}} : \mathcal{S} \textnormal{ a coherent subdivision of }(Q,A) \}
\end{align*}
forms a complete fan in $\mathbb{R}^A$. Furthermore, $C_{\mathcal{S}_2}$ is a face of $C_{\mathcal{S}_1}$ if and only if $\mathcal{S}_1$ refines $\mathcal{S}_2$. The fan $\mathcal{F}_{Sec A}$ is called the \textbf{secondary fan} of $A$.
\end{proposition}

A surprising fact about the fan $\mathcal{F}_{Sec A}$ is that it is also the normal fan of a polytope $\Sigma (A)$ called the \textbf{secondary polytope} of $(Q,A)$. As we are choosing the min convention, a normal fan here consists of cones $\{\sigma_F \}$ which are minimal on faces $F$ of $Q$. The secondary polytope has a rather pleasant explicit description. To provide this, let $\mathcal{T}= \{(Q_i, A_i)\}$ be a coherent triangulation, $\{e_a\}$ the basis vectors of the dual of $\mathbb{R}^A$, and $\text{Vol} (Q_i)$ the normalized volume of a simplex. Then we may define the vectors 
\begin{align}
    \varphi_{\mathcal{T}} = \sum_{a \in A} \left( \sum_{a \in Q_i} \textnormal{Vol} (Q_i) \right) e_a^* \in \left( \mathbb{R}^A \right)^*.
\end{align}

\begin{theorem}\cite{gkz}
The secondary polytope 
\[
\Sigma (A) = \convh \{\varphi_{\mathcal{T}} : \mathcal{T} \text{ a coherent triangulation of } (Q, A) \}
\]
has normal fan equal to $\mathcal{F}_{Sec A}$. Thus the face lattice of $\Sigma (A)$ is isomorphic to coherent subdivisions under the partial order of refinement.
\end{theorem}

By this theorem, one can identify the vertices of $\Sigma (A)$ with coherent triangulations of $(Q,A)$. This is illustrated for our two examples in Figure~\ref{fig:example1d2d}.

\subsection{Tropical hypersurfaces}
Another fruitful way of working with point configurations and polytopes is by considering the tropical geometry associated to $A$ and $\eta$ as discussed in \cite{macstur,britmish,gross}. This dual construction is sometimes referred to as the discrete Legendre transform of the data given in the previous section. Part of this is achieved by considering the \textbf{tropical polynomial} $f_\eta : V^* \to \mathbb{R}$ induced by $\eta$ which we write as 
\begin{align}
    f_{\eta} (u) := \min \{ u (a) + \eta (a) : a \in A \}.
\end{align}
This piecewise linear concave function induces a polyhedral decomposition $\mathcal{P}_\eta = \{ P_j \}$ of $V^*$ whose maximal cells are the maximal domains of linearity of $f_\eta$. For each such cell $P_j \in \mathcal{P}_\eta$ we define the marking of $P_j$ as 
\begin{align}
    \psi_\eta^A (P_j) := \{ a \in A : f_{\eta} (u) = u (a) + \eta (a) \textnormal{ for all } u \in P_j \}.
\end{align}
We codify these structures into a single definition.
\begin{definition}
Suppose $\mathcal{P}$ is a polyhedral subdivision of $V^*$, $f: V^* \to \mathbb{R}$ a piecewise affine function with corner locus on the codimension $1$ strata of $\mathcal{P}$ and $\psi : \mathcal{P} \to 2^A$ a marking of the cells of $\mathcal{P}$ by subsets of $A$. We call the data $(\mathcal{P} , f, \psi)$ a \textbf{tropical $A$-complex} if there exists some $\eta \in \mathbb{R}^A$ so that $(\mathcal{P} , f, \psi) = (\mathcal{P}_\eta , f_\eta, \psi^A_\eta)$.
\end{definition}
The following result establishes the link between coherent subdivisions and tropical $A$-complexes. 

\begin{proposition} \label{prop:duality}
For a given $\eta \in \mathbb{R}^A$, the face lattice of $\mathcal{S}_\eta$ is dual to the face lattice of $\mathcal{P}_\eta$. Furthermore, 
\begin{enumerate}
    \item the dual to a maximal marked polytope $Q_i$ of $\mathcal{S}_\eta$ which is the projection of the facet $F_i$ of $Q_\eta$ is the linear support function $g_{\eta, i}$, 
    \item the value of $f_\eta$ at $g_{\eta, i}$ is $c_{\eta, i}$,
    \item a face $(Q_i, A_i)$ of $\mathcal{S}_\eta$ is on the boundary of $Q$ if and only if the dual face $P_j$ is non-compact. 
\end{enumerate}
\end{proposition}
We omit the proof of this proposition, as it is immediate from the constructions and can be deciphered by considering the duality of the normal fan of $Q_\eta$. 

We write this correspondence as 
\begin{equation} \label{eq:dualiso}
	\begin{tikzpicture}[baseline=(current  bounding  box.center), scale=1.5]
	\node (A) at (0,0) {$\mathcal{S}_\eta$};
	\node (B) at (1,0) {$\mathcal{P}_\eta$};
	\path[->,font=\scriptsize]
	([yshift=-.7mm] A.north east) edge node[above] {$\phi^A$} ([yshift=-.7mm] B.north west);
	\path[->, font=\scriptsize, yshift=-.5em]
	(B) edge node[below] {$\psi^A$} (A);
	\end{tikzpicture} 
\end{equation}
where 
\begin{align}
    \phi^A (Q_i, A_i) & = \{u : f_\eta (u) = u (a_i) + \eta (a_i) \textnormal{ if } a_i \in A_i \}, 
\end{align}
and $\psi^A (P) = (Q_i, A_i)$ where $Q_i$ is the convex hull of $A_i$ and 
\begin{align}
    A_i & = \{a_i : f_\eta (u) = u (a_i) + \eta (a_i) \textnormal{ for all } u \in P \}.
\end{align}
On occasion, we may write $\psi^A (P) = A_i$ and omit reference to the convex hull.

If $\eta, \eta^\prime$ induce the same coherent subdivision $\mathcal{S}_\eta = \mathcal{S} = \mathcal{S}_{\eta^\prime}$, we say $\mathcal{P}_\eta$ is \textbf{isotopic} to $\mathcal{P}_{\eta^\prime}$. Indeed, one may connect $\eta$ to $\eta^\prime$ via a straight line curve in $C_{\mathcal{S}}^\circ$ and observe a topological isotopy of the tropical duals. Using Proposition~\ref{prop:duality}, if $\mathcal{P}$ is isotopic to $\mathcal{P}^\prime$, there is a unique isomorphism of face lattices
\begin{align} \label{eq:isotopy}
    \iota_{\mathcal{P}, \mathcal{P}^\prime} : \mathcal{P} \to \mathcal{P}^\prime
\end{align}
which takes the maximal dual to $a \in A$ in $\mathcal{P}$ to the maximal dual of $a$ in $\mathcal{P}^\prime$ (when $\{a\}$ is a face of the dual subdivision).
\begin{figure}[t]
\begin{tikzpicture}[cross line/.style={preaction={draw=white, -, line width=6pt}}]
    \node at (1,0) (node A) {};
    \node at (4,0) (node B){};
    \draw[thick]  (node A) -- (node B); 
    \node[inner sep=0] (image) at (-4,0) {\includegraphics[scale=.1]{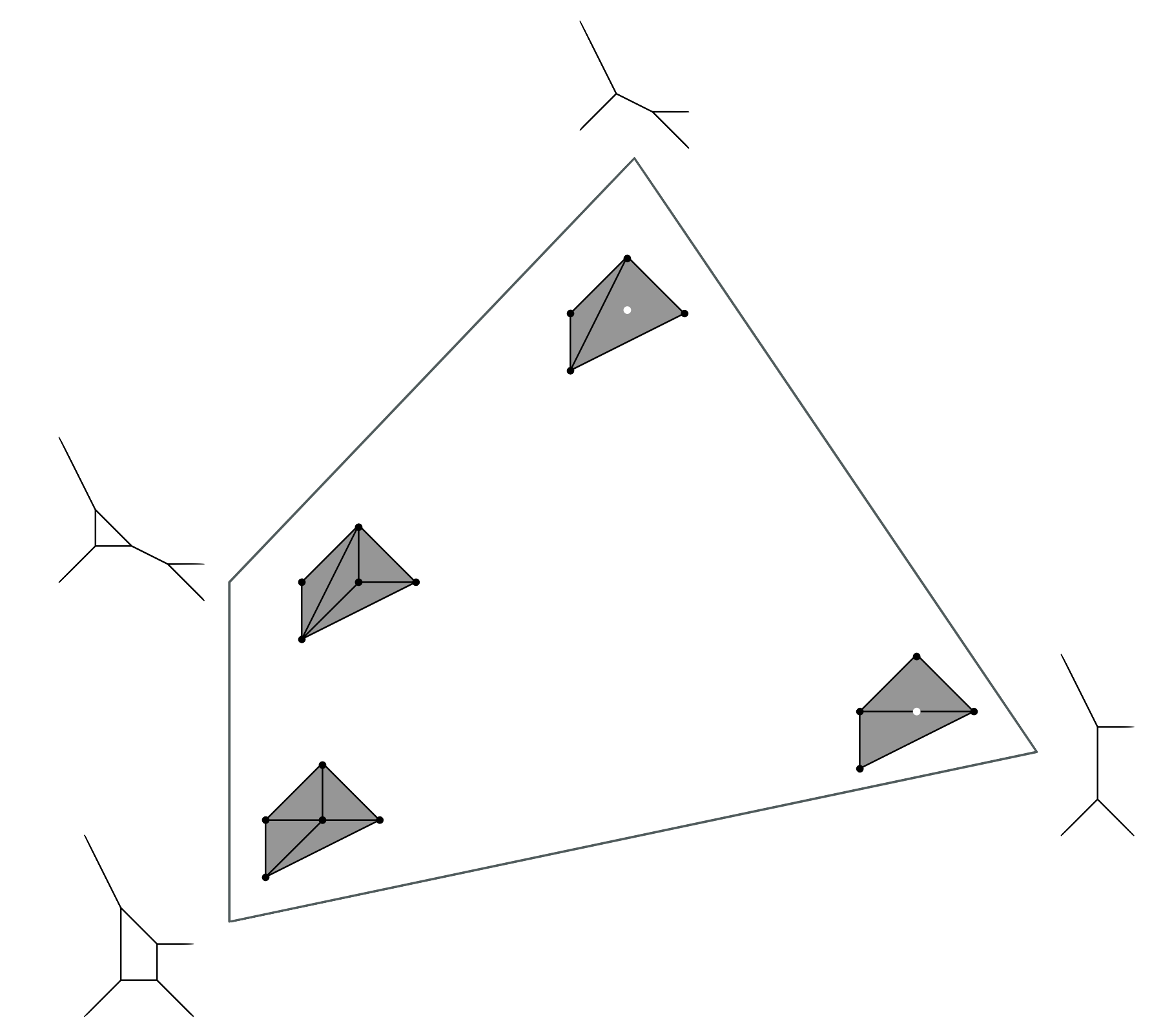}};
    \node[inner sep=0] (image) at (.75,-1) {\includegraphics[scale=.1]{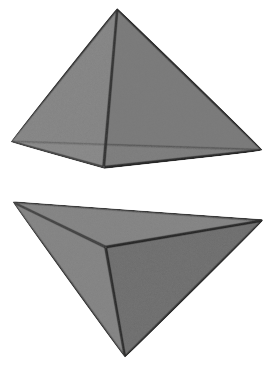}};
    \node[inner sep=0] (image) at (4.25,-1) {\includegraphics[scale=.1]{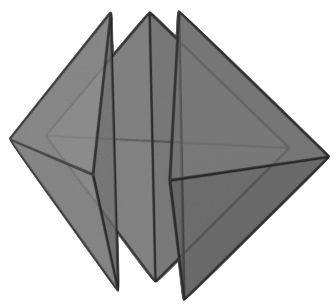}};
    \node[inner sep=0] (image) at (.75,1) {\includegraphics[scale=.1]{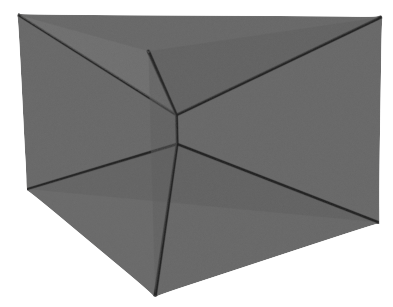}};
    \node[inner sep=0] (image) at (4.25,1) {\includegraphics[scale=.1]{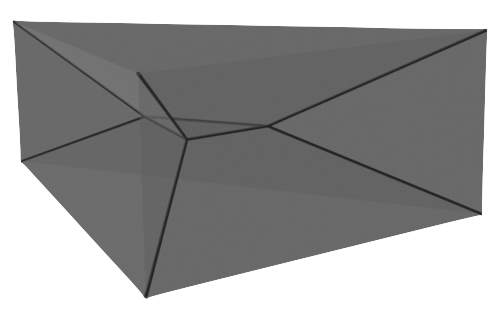}};
\end{tikzpicture}
\caption{\label{fig:example1d2d} The secondary polytopes of $(Q,A)$ and $(\tilde{Q}, \tilde{A})$ along with their corresponding coherent subdivisions and tropical hypersurfaces.}
\end{figure}
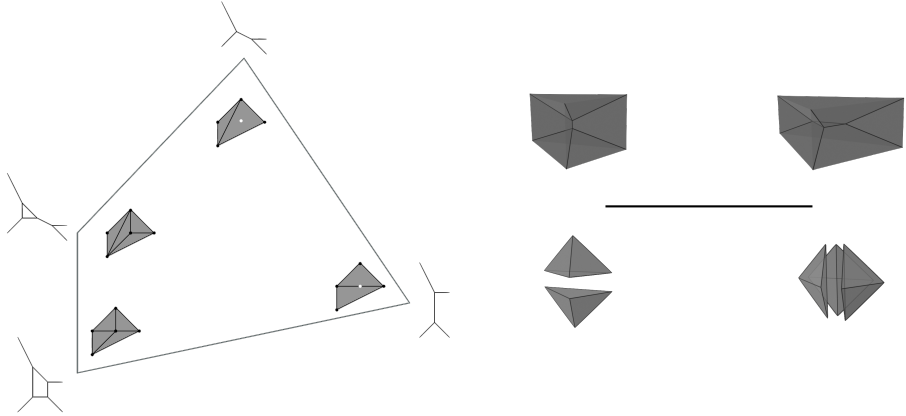

Note that the partial order $\preceq$ of refining a subdivision induces a partial order on isotopy classes of tropical $A$-complexes. 
\begin{proposition}
Let $\subdiv{A}$ be the poset of coherent subdivisions of $A$ and $\tropical{A}$ the poset of tropical $A$-complexes modulo isotopy, then $\mathcal{S}_\eta \mapsto \mathcal{P}_\eta$ induces a one-to-one correspondence 
\begin{equation} \label{eq:dualcorrespondence}
	\begin{tikzpicture}[baseline=(current  bounding  box.center), scale=2]
	\node (A) at (0,0) {$\subdiv{A}$};
	\node (B) at (1,0) {$\tropical{A}$};
	\path[->,font=\scriptsize]
	([yshift=-.7mm] A.north east) edge node[above] {$\Phi$} ([yshift=-.7mm] B.north west);
	\path[->, font=\scriptsize, yshift=-.5em]
	(B) edge node[below] {$\Psi$} (A);
	\end{tikzpicture} 
\end{equation}
between coherent subdivisions and isotopy classes of polyhedral $A$-complexes.
\end{proposition}
The proof of this proposition is evident from the fact that the marking $\psi^A$ of cells of $\mathcal{P}$ produces the coherent subdivision dual to $(\mathcal{P}, f, \psi)$. 

\begin{definition}
    The tropical hypersurface $\trop_\eta$ associated to $(A,\eta)$ is the codimension $1$ subcomplex of $\mathcal{P}_\eta$. 
\end{definition}

The secondary polytope thus can be seen as a dictionary whose face lattice encodes both coherent subdivisions of a polytope and, equivalently, their tropical duals. This is illustrated for our  examples in Figure~\ref{fig:example1d2d}. A well known example of this circle of ideas is the case when $(Q,A)$ is an $n$-gon in the plane $V \cong \mathbb{R}^2$ which will be discussed in Section~\ref{sec:multiplihedra}. 

\begin{figure}[t]
\begin{tikzpicture}[cross line/.style={preaction={draw=white, -, line width=6pt}}]
     \node[anchor=south west,inner sep=0] (image) at (0,0) {\includegraphics[scale=.25]{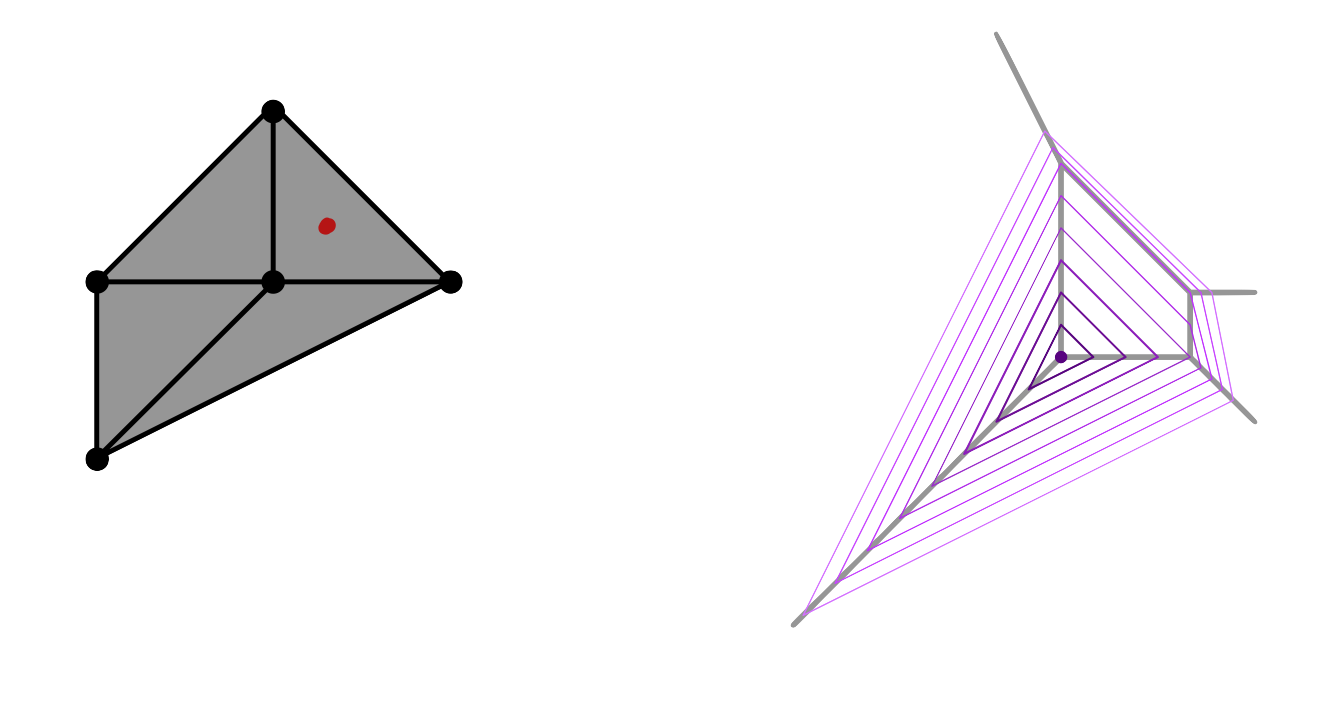}};
\end{tikzpicture}
\caption{\label{fig:example1e} A contour graph of $f_\eta - \alpha$.}
\end{figure}

\section{Painted tropical complexes}
For what follows, we will fix an element $\alpha$ in $V$. 
\begin{definition}
A \textbf{painted tropical $A$-complex} $(\mathcal{P}, f, \psi, \kappa)$ relative to $\alpha \in V$ is a tropical $A$-complex $(\mathcal{P}, f, \psi)$ along with a color function
\begin{align}
    \kappa : \mathcal{P} \to \{r, p, b\}
\end{align} 
for which there exists an $\eta \in \mathbb{R}^A$ and $c \in \mathbb{R}$ so that 
\begin{enumerate}
    \item $(\mathcal{P}, f, \psi) = (\mathcal{P}_\eta, f_\eta, \psi^A)$,
    \item the color function $\kappa$ equals $\kappa_{\eta, c}$ is defined by 
    \begin{align*}
    \kappa_{\eta, c} (\sigma) = \begin{cases} b & f_\eta (u) <  u (\alpha ) + c \textnormal{ for all } u \in \relint (\sigma ), \\ p & f_\eta (u) =  u (\alpha ) + c \textnormal{ for some } u \in \relint (\sigma ), \\ r & f_\eta (u) >  u (\alpha ) + c \textnormal{ for all } u \in \relint (\sigma ). \end{cases}
\end{align*}
\end{enumerate}
\end{definition}
For a painted tropical $A$-complex $(\mathcal{P}, f, \psi, \kappa)$ , we call cells $\sigma$ with $\kappa (\sigma ) = b$ (respectively $p$, $r$) \textbf{blue} (respectively \textbf{purple}, \textbf{red}) cells. We say that the painted tropical $A$-complexes $(\mathcal{P}_1, f_1, \psi_1, \kappa_1)$ and $(\mathcal{P}_2, f_2, \psi_2, \kappa_2)$ are isotopic if the tropical $A$-complexes are isotopic and the isotopy preserves the color functions.
\begin{example} \label{ex:qwi2}
Take $(Q, A)$ to be the case of the quadrilateral with an interior point given in Example~\ref{ex:qwi0}. To obtain a painted tropical $A$-complex relative to $\alpha$, we must choose an $\eta, c$ and $\alpha$. The choice of $\eta$ will simultaneously give a subdivision of $(Q,A)$ and a tropical $A$-complex. On the left of  Figure~\ref{fig:example1e} we illustrate the triangulation associated to $\eta = -e_{(0,0)}$.  Choosing $\alpha = (1/3, 1/3)$, the right side of the figure describes the contour graph of $f_\eta - \alpha$ with the tropical hypersurface in the background. In Figure~\ref{fig:example1f}, the painted tropical $A$-complexes are shown for increasing values of $c$ with $\eta$ and $\alpha$ fixed.
\begin{figure}[ht]
\begin{tikzpicture}[cross line/.style={preaction={draw=white, -, line width=6pt}}]
     \node[anchor=south west,inner sep=0] (image) at (0,0) {\includegraphics[scale=.23]{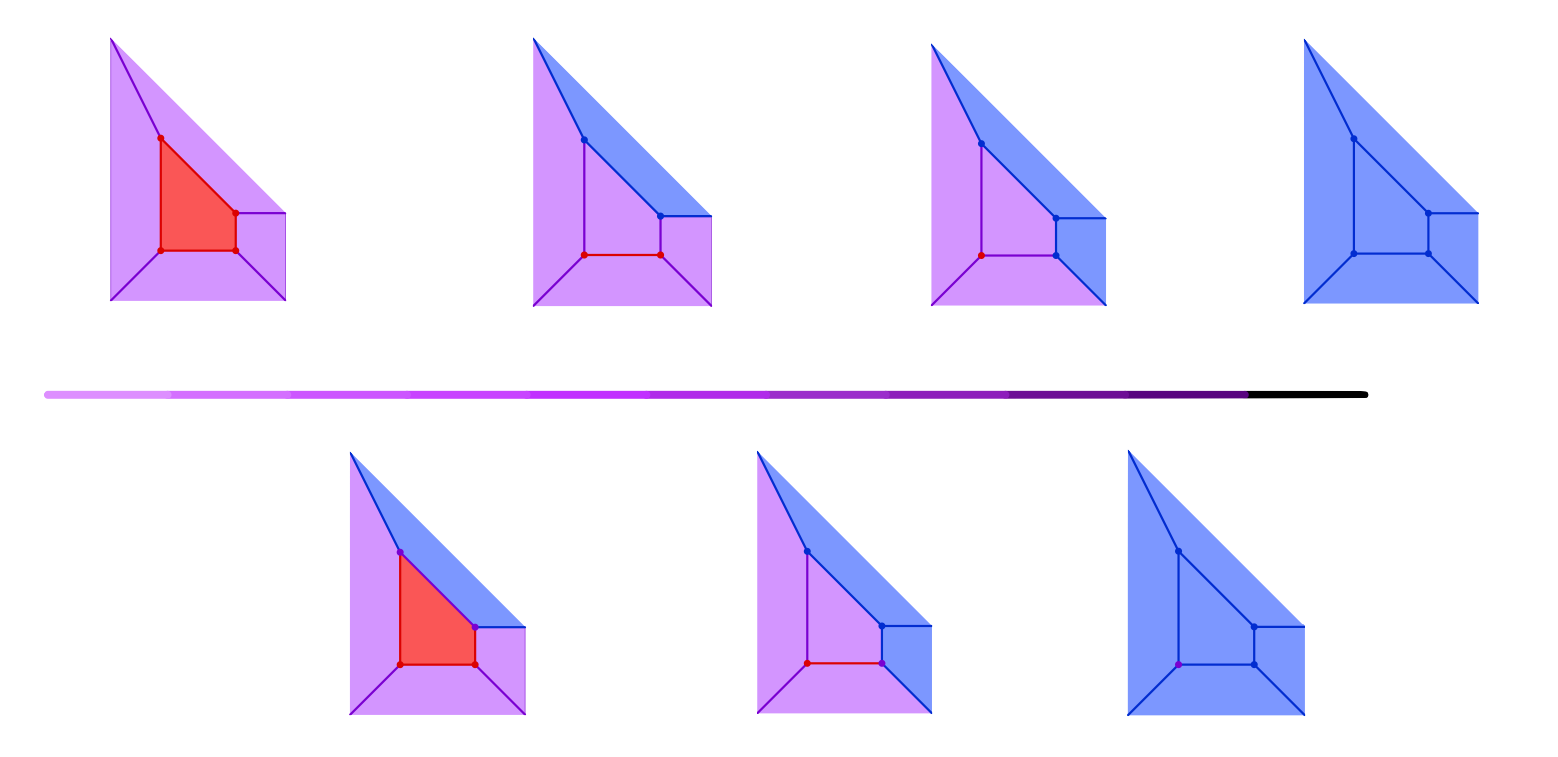}};
\end{tikzpicture}
\caption{\label{fig:example1f} Painted tropical $A$-complexes.}
\end{figure}
\end{example}

By Proposition~\ref{prop:duality}, the non-compact cells of a tropical $A$-complex $\mathcal{P}_\eta$ are dual to faces of $Q$. In fact, the coloring of these cells has certain restrictions based on the position of $\alpha$. To explain this, we introduce a bit of notation.  Let $\{\nu_1, \ldots, \nu_m\} \subset V^*$ be the supporting functions of the facets of $Q$ so that
\begin{align} \label{eq:support}
Q = \bigcap_{i = 1}^m \left\{x : \nu_i (x) \geq d_i \right\}.
\end{align}
Write $F_i$ for the facet $\{x \in Q : \nu_i (x) = d_i\}$. The distinguished element $\alpha$ defines a sign vector $\mathbf{s}_{A, \alpha} = (s_1, \ldots, s_m) \in \{-, 0, +\}^m$ determined by taking
\begin{align} \label{eq:signvec}
    s_i = \begin{cases} + & \nu_i (\alpha ) < d_i, \\ 0 & \nu_i (\alpha ) = d_i, \\ - & \nu_i (\alpha ) > d_i. \end{cases}
\end{align}

\begin{proposition} \label{prop:rays}
Let $\sigma_i$ be the non-compact $1$ dimensional cell of a tropical $A$-complex whose dual facet is $F_i = \psi (\sigma_i)$. Then, outside of any sufficiently large bounded subset $K$ of $\sigma_i$, 
\begin{enumerate}[topsep=3pt,itemsep=1ex]
\item $f_\eta (u) > u (\alpha ) + c$ for all $u \in \sigma_i - K$ if $s_i = +$,
\item $f_\eta (u) < u (\alpha ) + c$ for all $u \in \sigma_i - K$ if $s_i = -$,
\end{enumerate}
If $s_i = 0$ then $u (\alpha ) + c - f_\eta (u)$ is constant on $\sigma$.
\end{proposition}
\begin{proof}
To see this, note first that $\sigma_i = u_0 + \mathbb{R}_{\geq 0} \cdot \nu_i$ for some $u_0 \in V^*$. This follows from an elementary argument, or the fact that the recession fan of the tropical $A$-complex is the normal fan of $Q$. 

Now, let $u \in \sigma_i$ be any element and write it as $u_0 + r \nu_i$ for some $r \geq 0$.  Because $\psi (\sigma_i ) = F_i$ we must have that, for any $a \in F_i$,
\[
f_\eta (u) = u(a) + \eta (a) = u_0 (a) + r d_i + \eta (a).
\]
On the other hand,  
\[
u (\alpha ) = (u_0 + r \nu_i ) (\alpha ) = u_0 (\alpha ) + r \nu_i (\alpha ).
\]
If $a_i = 0$ so that $\nu_i (\alpha ) = d_i$, then subtracting gives $f_\eta (u) - u(\alpha)$ is constant on $\sigma_i$. On the other hand, if $a_i \ne 0$ then $\nu_i (a) \ne d_i$ so we may define
\[
R = \frac{u_0 (a) + \eta (a) - c - u_0 (\alpha ) }{\nu_i (\alpha ) - d_i}.
\]
If $a_i = +$, then for $r > R$ we have 
\[
(d_i - \nu_i (\alpha )) r > c + u_0 (\alpha ) - u_0 (a) - \eta (a)  
\]
and so
\begin{align*}
 f_\eta (u) - u(\alpha ) - c & = u_0 (a)  + r d_i + \eta (a) -  u_0 (\alpha ) - r \nu_i (\alpha ) - c, \\ & = u_0 (a)   + \eta (a) -  u_0 (\alpha ) - c +   (d_i - \nu_i (\alpha ))r, \\ & > 0 ,
\end{align*}
The opposite inequality for $a_i = -$ is found in a similar manner.
\end{proof}

\begin{proposition} \label{prop:zerosimplices}
A painted tropical $A$-complex relative to $\alpha \in V^*$ is determined by the tropical $A$-complex $\mathcal{P}$ and painting the $0$-cells.
\end{proposition}
\begin{proof}
To see this, one need only show that, given a painted tropical $A$-complex  $(\mathcal{P}, f, \psi)$ relative to $\alpha$ with a color assignment $\kappa$, then $\kappa$ can be recovered from its image on $0$-cells $\kappa (0) : \mathcal{P} (0) \to \{r, p, b\}$.

Suppose $\eta$ and $c$ determine $\kappa$ and define the difference function $g : V^* \to \mathbb{R}$ by
\[
g (u) = f_\eta (u) - u (\alpha ) - c.
\]
Then for $\sigma$ in $\mathcal{P}$, the coloring of $\sigma$ is blue (resp. red) if $g < 0$ (resp. $g > 0$) on the relative interior of $\sigma$, and purple if $g = 0$ at some point on the relative interior.

There are two cases to consider. Assume $\sigma \in \mathcal{P}$ is a compact cell. Then it is the convex hull of its $0$-dimensional faces. As $g$ is an affine function when restricted $\sigma$, it is negative (resp. positive) on the relative interior of $\sigma$ if and only if it is negative (resp. positive) on all vertices except those lying on a single face where it may be zero. If this does not occur, then either all vertices of sigma have $g$ value $0$ implying $g$ is identically zero on $\sigma$ or there are both positive and negative $g$ values on the vertices, implying $g$ attains a zero on the relative interior of $\sigma$.

Another case to take into consideration is when $\sigma$ is non-compact. Now, if $\sigma$ is a ray, then its dual $\psi (\sigma )$ is a facet $F_i$ and by Proposition~\ref{prop:rays}, the sign of $g$ outside of a compact set is determined by $s_i$ if $s_i \ne 0$. If this sign is opposite of that of the vertex of $\sigma$, then $\sigma$ must achieve a zero. Otherwise the whole ray must have the sign $s_i$. If $s_i = 0$ then again Proposition~\ref{prop:rays} gives that $g$ is constant on $\sigma$ so that its sign is determined by the sign of $g$ on the ray's vertex.

If $\sigma$ is a non-compact cell that is not a ray, then its color can be determined by examining the colors of the faces that are compact or rays as in the compact case.
\end{proof}

Given $(\eta_1, c_1)$ and $(\eta_2, c_2)$ in $\mathbb{R}^A \times \mathbb{R}$ we say $(\eta_1, c_1) \sim (\eta_2, c_2)$ if the painted tropical $A$-complexes $(\mathcal{P}_{\eta_1}, f_{\eta_1}, \psi_{\eta_1}^A, \kappa_1)$ and $(\mathcal{P}_{\eta_2}, f_{\eta_2}, \psi_{\eta_2}^A, \kappa_2)$ are isotopic and the isotopy respects the color functions. Write 
\[
C_{\eta, c}^\circ = \left\{ (\tilde{\eta}, \tilde{c}) \in \mathbb{R}^A \times \mathbb{R} : (\eta , c ) \sim (\tilde{\eta}, \tilde{c} ) \right\}
\]
and $C_{\eta , c}$ its closure. 

\begin{proposition}
For any $(\eta , c)$, $C_{\eta, c}$ is a polyhedral cone and the collection $\mathcal{F}_{A, \alpha} = \{C_{\eta, c} \}$ forms a complete fan supported in $\mathbb{R}^A \times \mathbb{R}$.
\end{proposition}
\begin{proof}
The fact that $C^\circ_{\eta, c}$ is a cone follows from the fact that if $\mathcal{S} = \mathcal{S}_\eta$ then $C^\circ_{\eta, c}$ projects to $C^\circ_{\mathcal{S}}$ which is a (relatively open) polyhedral cone in $\mathcal{F}_{Sec A}$. The remaining inequalities defining $C^\circ_{\eta , c}$ can be written using the color function $\kappa$ in the painted tropical $A$-complex $(\mathcal{P}, \mathcal{P}^r)$. This can be done as follows. 
For a zero dimensional face $\{p\}$ of $\mathcal{P}$  the correspondence from Proposition~\ref{prop:duality} gives a subset $\psi (\{p\}) = \{a^p_1, \ldots, a^p_{k_p}\}$. Here, the elements $a^p_1, \ldots, a^p_{k_p}$ affinely span $V$ (because $\{p\}$ is a zero dimensional face of $\mathcal{P}$). Thus there are constants $b_1^p, \ldots, b_{k_p}^p$ so that 
\begin{align*}
    \sum_{i = 1}^{k_p} b_i^p & = 1, \\
    \sum_{i = 1}^{k_p} b_i^p a_i^p & = \alpha. 
\end{align*}
Then, for any $1 \leq i \leq k_p$ define the linear function $b_p : \mathbb{R}^A \times \mathbb{R} \to \mathbb{R}$ via
\begin{align}
    b_p (\tilde{\eta}, \tilde{c}) & = - \tilde{c} +  \sum_{i = 1}^{k_p} b_i^p \tilde{\eta} (a_i^p)  .
\end{align}
Observe that for any $\tilde{\eta} \in C_S^\circ$ there is a zero cell $\tilde{p} \in \mathcal{P}_{\tilde{\eta}}$ which is isotopic to $p$. Then we have $\phi (\{\tilde{p}\})$ contains $\{a_1^p, \ldots , a_{k_p}^p\}$. This means that for every $1 \leq j \leq k$ we have
\begin{align*}
    f_{\tilde{\eta}} (\tilde{p} ) & = \tilde{p} ( a_j^p) + \tilde{\eta} (a_j^p),
\end{align*} 
so that 
\begin{align*} f_{\tilde{\eta}} (\tilde{p} ) & = \sum_{i = 1}^{k_p} b_i^p f_{\tilde{\eta}} (\tilde{p} ), \\ & = \sum_{i = 1}^{k_p} b_i^p \tilde{p} ( a_i^p) + \sum_{i = 1}^{k_p} b_i^p \tilde{\eta} (a_i^p), \\ & = \tilde{p} \left( \alpha \right) + \tilde{c} + b_p (\tilde{\eta} , \tilde{c} ) .
\end{align*}
From this, it is clear that specifying $b_p < 0$, $b_p = 0$, or $b_p > 0$ if $p$ is red, purple or blue respectively, will describe those pairs $(\tilde{\eta}, \tilde{c} )$ with $\tilde{p}$ sharing the color of $p$. Thus, assembling these inequalities and constraints for all zero cells in $\mathcal{P}$ describes $C_{\eta, c}^\circ$ as a polyhedral cone. Taking intersections of closures of these cones clearly converts some inequalities into constraints thus giving another cone in $\mathcal{F}_{A, \alpha}$.
\end{proof}

Our main result is that the fan $\mathcal{F}_{A, \alpha}$ is combinatorially equivalent to a secondary fan (and in fact linearly equivalent after quotienting by a lineality subspace). To show this, consider extending $A$ by a placing $V$ as a zero hyperplane and adding two new vertices to obtain
\begin{align}
\bar{A}_\alpha = \left\{ (a, 0) \in V \times \mathbb{R}: a \in A \right\} \cup \{(\alpha , 1), (\alpha , 2) \} \subset V \times \mathbb{R}.
\end{align}
We will use the suggestive notation
\begin{align}
    \rho & = (\alpha , 1), \\
    \beta & = (\alpha , 2).
\end{align}
Additionally, given a fan $\mathcal{F}$ supported in a vector space $W$ we will write $\mathcal{F} \times \mathbb{R}^n$ for the fan with cones $\sigma \times \mathbb{R}^n$ in $W \times \mathbb{R}^n$. 

\begin{theorem} \label{thm:mainthm}
The fan $\mathcal{F}_{A, \alpha} \times \mathbb{R}$ is isomorphic to $\mathcal{F}_{Sec \bar{A}_\alpha}$.
\end{theorem}
\begin{proof}
To prove this, we first consider the evaluation map associated to $\bar{A}_\alpha$. Let $\{e_{\bar{a}} : \bar{a} \in \bar{A}_\alpha \}$ be the standard basis of $\mathbb{R}^{\bar{A}_\alpha}$ and consider the linear map 
\begin{align*}
 \textnormal{ev}: \mathbb{R}^{\bar{A}_\alpha} \to V \times \mathbb{R} \times \mathbb{R}
\end{align*}
defined as
\begin{align*}
    \textnormal{ev} (e_{\bar{a}} ) = (\bar{a}, 1).
\end{align*}
Note that for $\bar{a} = (a, 0)$ we have $\textnormal{ev}(e_{\bar{a}} ) = (a, 0, 1)$ whereas 
\begin{align*}
    \textnormal{ev}(e_{\rho} ) & = (\alpha, 1, 1), \\
    \textnormal{ev}(e_{\beta} ) & = (\alpha, 2, 1).
\end{align*}
The secondary fan $\mathcal{F}_{Sec \bar{A}_\alpha}$ is naturally supported in the dual vector space $\left( \mathbb{R}^{\bar{A}_\alpha} \right)^*$ which we identify with real valued functions on $\bar{A}_\alpha$ (i.e. with $\mathbb{R}^{\bar{A}_\alpha}$ but using the dual basis). Recall that for any fan $\mathcal{F}$, the lineality space of $\mathcal{F}$ is the minimal cone in $\mathcal{F}$ or, equivalently, the maximal linear subspace of vectors which stabilize the fan under translation. Then a basic result from \cite{gkz}[Proposition~7.1.11] identifies the lineality space of $\mathcal{F}_{Sec \bar{A}_\alpha}$ as the image of the dual map 
\begin{align*}
  L := \textnormal{ev}^* (V^* \times \mathbb{R} \times \mathbb{R}) \subset \mathbb{R}^{\bar{A}_\alpha}.
\end{align*}

We note this because, if a vector subspace $U$ in $\mathbb{R}^N$ is transverse to the lineality space $L_\mathcal{F}$ of a fan $\mathcal{F}$, and $L_U$ is a subspace of $L_\mathcal{F}$ complementary to $U$, then it is easy to see that $\mathcal{F} \approx \mathcal{F}|_U \times L_U$. Here $\mathcal{F}|_U$ is the set of intersections of cones in $\mathcal{F}$ with the subspace $U$. Thus, the statement of the theorem follows if we produce an inclusion
\begin{align*}
    i : \mathbb{R}^A \times \mathbb{R} \to \mathbb{R}^{\bar{A}_\alpha}
\end{align*}
which is transverse to $L$ and which maps $\mathcal{F}_{A, \alpha}$ isomorphically to $\mathcal{F}_{Sec \bar{A}_\alpha}|_{\textnormal{im} (i )}$. Define this map as
\begin{align}
    i (\eta , c) = c e^*_\rho + c e^*_\beta + \sum_{a \in A} \eta (a) e^*_{(a, 0)}.
\end{align}
Note that $i$ is injective and the vector $\textnormal{ev}^* (0 , 1, 0) = 2 e_\beta^* + e_\rho^* \in L - \textnormal{im} (i )$ so that $L + \textnormal{im} (i ) = \mathbb{R}^{\bar{A}_\alpha}$ and thus $i$ is transverse to $L$.

Now, let $(\mathcal{P}, \kappa)$ be a painted tropical $A$-complex realized by $(\eta , c)$ so that $\mathcal{P}$ is isotopic to $\mathcal{P}_\eta$ and $\kappa$ is compatible with $\kappa_{\eta, c}$. Taking $\xi = i (\eta , c)$ to be its image, we claim that the tropical $A$-complex $\mathcal{P}_\xi$ is determined uniquely up to isotopy by $(\mathcal{P}, \kappa)$ so that 
\[
i (C^\circ_{\eta, c} ) = C^\circ_{\mathcal{S}_\xi} \cap \textnormal{im} (i). \]

To verify the claim, consider the tropical $A$-complex $(\mathcal{P}_\xi, f_\xi)$ where $f_\xi : V^* \times \mathbb{R} \to \mathbb{R}$ is the tropical polynomial. For $\bar{u} = (u, d) \in V^* \times \mathbb{R}$ we have 
\begin{align} \label{eq:fxi}
f_\xi (\bar{u}) & = \min \{\bar{u}(\bar{a}) + \xi (\bar{a}) : \bar{a} \in \bar{A}_{\alpha }\}, \\ \nonumber
& =  \min \left(\{u (a) + \eta (a) : a \in A\} \cup \{u (\alpha ) + d + c, u (\alpha ) + 2d  + c \} \right) , \\ \nonumber & =  \min \{f_\eta (u), u (\alpha ) + d + c, u (\alpha ) + 2d  + c \} .
\end{align}

Now, let $u \in V^*$ be a vertex of $\mathcal{P}_\eta$ which is dual to the set $\psi (u) \subset A$.  Note that for any such vertex, there is a unique vertex $\bar{u} = (u , d)$ in $\mathcal{P}_\xi (0)$ with first coordinate $u \in V^*$. Indeed, we define $d$ via the equation
\begin{align}
    d = \begin{cases} f_\eta (u) - u(\alpha ) - c & \textnormal{if } f_\eta (u) \geq u(\alpha ) + c, \\ \frac{1}{2} (f_\eta (u) - u(\alpha ) - c ) & \textnormal{if } f_\eta (u) \leq u(\alpha ) + c. \end{cases}
\end{align}
In the case where $f_\eta (u) > u (\alpha ) + c$ we have $d > 0$ implying that 
\[
\bar{u} (\rho ) + \xi (\rho ) = u (\alpha ) + d + c < u (\alpha ) + 2 d + c = \bar{u} (\beta ) + \xi (\beta ) 
\]
and 
\[
\bar{u} (\rho ) + \xi (\rho ) = u (\alpha ) + d + c = f_\eta (u)
\]
so that the minimum of $f_\xi$ is achieved at all $(a, 0)$ for $a \in \psi (u)$ and $\rho$. Take $\sigma$ to be the minimal cell of $\mathcal{P}_\xi$ containing $\bar{u}$ so that 
\begin{align}
    \psi^{\bar{A}_\alpha} (\sigma ) = \{\rho , (a, 0) : a \in \psi^A (u) \}.
\end{align} 
Since $\rho$ is not in $V$ and $u$ is a vertex of $\mathcal{P}_\eta$, the convex hull of $\psi^{\bar{A}_\alpha} (\bar{u})$ is full dimensional implying $\bar{u} = (u,d)$ is a vertex of $\mathcal{P}_\xi$.

To see that $(u,d)$ is unique amongst vertices of $\mathcal{P}_\xi$ with first coordinate $u$,  consider $\bar{\bar{u}} = (u, d^\prime)$ for some $d^\prime \ne d$. If $d^\prime > d$, then 
\[
\bar{\bar{u}} (\beta ) + \xi (\beta) > \bar{\bar{u}} (\rho) + \xi (\rho ) = f_\eta (u) + (d^\prime - d) > f_\xi (\bar{\bar{u}})
\]
so that $f_\xi$ achieves its minimum at the vertices $(a, 0)$ where $a \in \psi^A (u)$. As the convex hull of these vertices is not full dimensional, we see that $\bar{\bar{u}}$ is not a vertex. On the other hand, if $d^\prime < d$ then
\[
\bar{\bar{u}} (\rho) + \xi (\rho ) = f_\eta (u) + (d^\prime - d) < f_\eta (u).
\]
So $f_\xi (\bar{\bar{u}} )$ achieves its minimum on a subset of $\{\rho , \beta\}$ which, again, is not full dimensional. 

In the case where $f_\eta (u) < u (\alpha ) + c$ we have $d < 0$ and a similar argument shows that 
\begin{align}
    \psi^{\bar{A}_\alpha} (\bar{u} ) = \{\beta , (a, 0) : a \in \psi^A (u) \}.
\end{align} 
Finally, when we have equality and $f_\eta (u) = u(\alpha) + c$ we obtain $d = 0$ and 
\begin{align}
    \psi^{\bar{A}_\alpha} (\bar{u} ) = \{\rho, \beta , (a, 0) : a \in \psi^A (u) \}.
\end{align} 

Thus we obtain that every vertex of $\mathcal{P}_\eta$ yields a unique vertex of $\mathcal{P}_\xi$ and the dual maximal cell of the coherent subdivision $\mathcal{S}_\xi$ is determined by the coloring function $\kappa_{\eta, c} : \mathcal{P} (0) \to \{r, p, b\}$. 

To determine the remaining vertices of $\mathcal{P}_\xi$ we again appeal to a dimension argument. Suppose $\bar{u} = (u, d)$ is a vertex of $\mathcal{P}_\xi$, where $u$ is not a vertex of $\mathcal{P}_\eta$. Let $\sigma$ be the minimal cell of $\mathcal{P}_\eta$ containing $u$ in its relative interior. Because $\bar{u}$ is a vertex, equation~\ref{eq:fxi} implies $f_\xi (\bar{u} ) = f_\eta (u)$ and so
\[
\left\{ (a, 0) : a \in \psi^A (\sigma) \right\} \subset \psi^{\bar{A}_\alpha} (\bar{u}).
\]
As $\sigma$ is not $0$-dimensional, $\psi^A (\sigma)$ has codimension $\geq 1$ in $V$. But if its codimension is $> 1$ in $V$ then it is $> 2$ in $V \times \mathbb{R}$ which would imply that, even upon adding both $\rho$ and $\beta$, $\psi^{\bar{A}_\alpha} (\bar{u})$ would not be full dimensional. Thus $\psi^A (\sigma )$ must have codimension $1$, so that $\sigma$ is $1$-dimensional. Furthermore, since $\psi^A (\sigma )$ is codimension $2$ in $V \times \mathbb{R}$, both $\rho$ and $\beta$ must be included to obtain a full dimensional dual so that 
\[
\left\{ (a, 0) : a \in \psi^A (u) \right\} \cup \{\rho , \beta \} = \psi^{\bar{A}_\alpha} (\bar{u}).
\]
This, along with equation~\ref{eq:fxi} implies
\[
u (\alpha ) + d + c = \bar{u} (\rho ) + \xi (\rho ) = f_\eta (u) =  \bar{u} (\beta ) + \xi (\beta ) = u (\alpha ) + 2 d + c
\]
so that $d = 0$ and
\[
u(\alpha ) + c = f_\eta (u).
\]
The difference $f_\alpha - \alpha^* - c$ is affine on the one dimensional polyhedra $\sigma$, but if it is constantly zero on $\sigma$, then $\{(v, 0) : v \in \sigma\}$ is a cell of $\mathcal{P}_\xi$ containing $\bar{u}$ contradicting that $\bar{u}$ is a vertex. Thus it is not constant on $\sigma$ and  there exists at most one zero on its relative interior. Such cells in $\mathcal{P}(1)$ are precisely those for which $\kappa (\sigma ) = p$, and for which $\kappa$ is not identically $p$ on the boundary of $\sigma$. 

Altogether, the coherent subdivision $\mathcal{S}_\xi$ is uniquely determined by its (marked) maximal cells. In turn, the maximal cells of the marked subdivision $\mathcal{S}_\xi$ are determined by the zero dimensional cells of $\mathcal{P}_\xi$. Finally, these cells are uniquely determined by the painted tropical $A$-complex $(\mathcal{P}, \kappa)$, verifying the claim and the theorem.
\end{proof}

\begin{figure}[ht]
\begin{tikzpicture}[line width=2pt,line cap=rect]
    \draw[thick]  (3.5,0) -- (3.5,2);
    \draw[thick]  (3.5,2) -- (3,4);
    \draw[thick]  (3,4) -- (1.5,6.2);
    \draw[thick]  (1.5,6.2) -- (-1,6.2);
    \draw[thick]  (-1,6.2) -- (-2,3);
    \draw[thick]  (-2,3) -- (-2,0);
    \draw[thick]  (-2,0) -- (3.5,0);
    \node[inner sep=0] (image) at (-3.5,0) {\includegraphics[scale=.15]{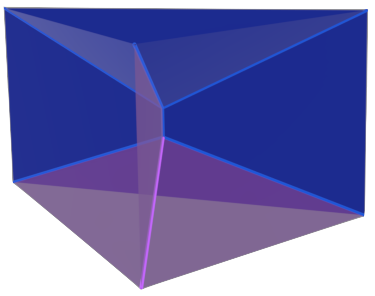}};
    \node[inner sep=0] (image) at (-3.5,3) {\includegraphics[scale=.15]{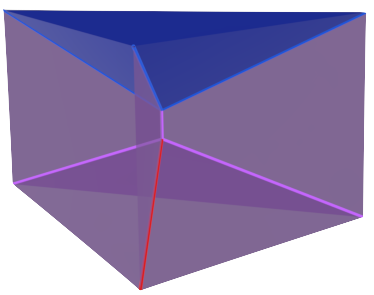}};
    \node[inner sep=0] (image) at (-2.5,6.5) {\includegraphics[scale=.15]{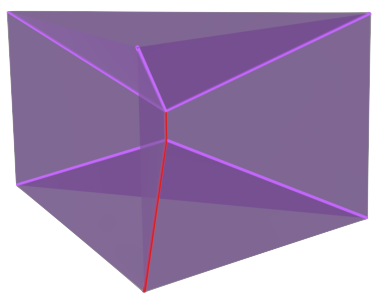}};
    \node[inner sep=0] (image) at (5.25,0) {\includegraphics[scale=.15]{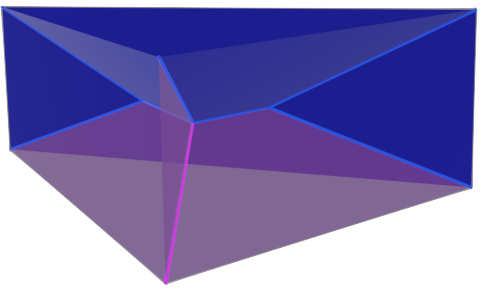}};
    \node[inner sep=0] (image) at (5.25,2.35) {\includegraphics[scale=.15]{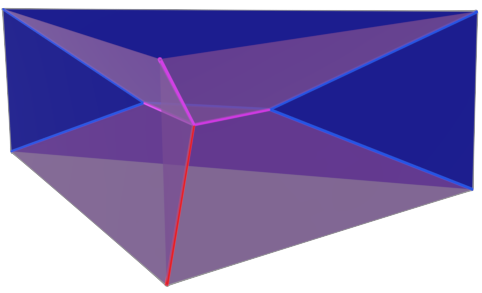}};
    \node[inner sep=0] (image) at (4.75,4.5) {\includegraphics[scale=.15]{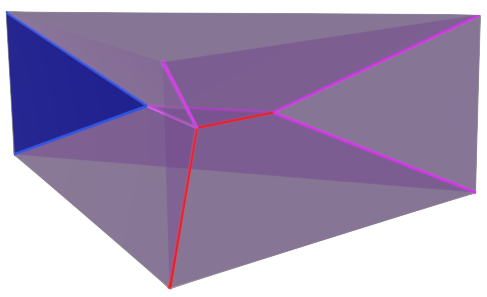}};
    \node[inner sep=0] (image) at (3.25,6.5) {\includegraphics[scale=.15]{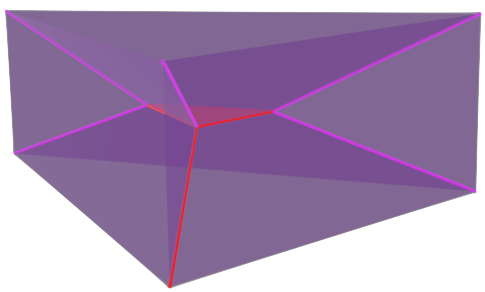}};
\end{tikzpicture}
\caption{\label{fig:paintingpolytope} The painting polytope of $(\tilde{Q},\tilde{A})$ by $\alpha$ along with the corresponding painted tropical $\tilde{A}$-complexes.}
\end{figure}
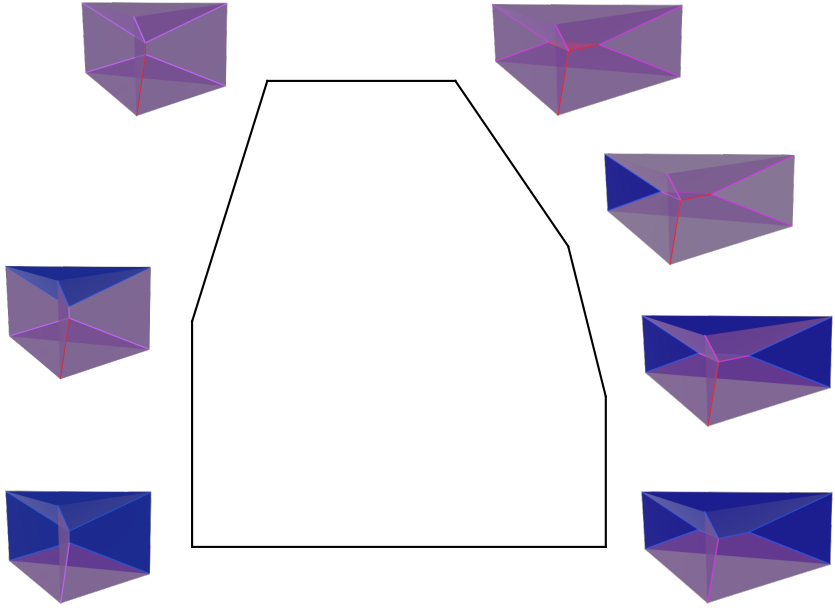
This theorem motivates the following definition.
\begin{definition}
The polytope $\Sigma (\bar{A}_\alpha )$ is called the \textbf{painting polytope} of $A$ by $\alpha$.
\end{definition}
\begin{example}
Recall that in Example~\ref{ex:bp0} we considered the circuit $(\tilde{Q}, \tilde{A})$ illustrated in Figure~\ref{fig:example1a2a}. We consider $\alpha = (1/2,1/3,1/2)$ which lies outside of $\tilde{Q}$, but near the facet $F$ with vertices $(1,0,0), (0,1,0), (0, 0, 1)$. Because if this location, Proposition~\ref{prop:rays} implies that the tropical ray dual to $F$ will be either purple or red, and all other rays will be either purple or blue. The painting polytope for $\tilde{A}$ by $\alpha$ along with the corresponding painted tropical $A$-complexes is illustrated in Figure~\ref{fig:paintingpolytope}.
\end{example}

\section{Multiplihedra as secondary polytopes} \label{sec:multiplihedra}

Our main application of Theorem~\ref{thm:mainthm} is to exhibit multiplihedra as a secondary polytope. First we fix an integer $k > 2$ take $Q$ to denote the $(m + 1)$-gon in $\mathbb{R}^2$ with vertices $A = \{a_0, \ldots, a_m \}$, ordered counter-clockwise, as its marking. Recall from \cite{gkz} that the secondary polytope of $(Q, A)$ is combinatorially equivalent to the associahedron $K_{m}$. For the remainder of this section, we will make this relationship explicit by relating a tropical $A$-complex $\mathcal{P}$ with the rooted planar tree formed by its $0$ and $1$-dimensional cells. 
The root of the tree is the dual ray of the edge from $a_m$ to $a_0$ of $Q$ and we direct the edges of the tree from the root to the leaves. To work with these trees, we write $\mathcal{P}^{cpt} (1)$ for the set of compact edges of $\mathcal{P}$. For any $e \in \mathcal{P}^{cpt} (1)$, denote its boundary points by $p_e$ and $q_e$ so that $e$ is directed from $p_e$ to $q_e$. Now, we will use the following lemma. 
\begin{lemma} \label{lemma:distance}
Let $\mathcal{P}$ be a tropical $A$-complex dual to $(Q,A)$. Then, for any $\beta \in \mathbb{R}^2$ not lying on the affine hull of a diagonal of $Q$, and any map 
\begin{align}
    \ell : \mathcal{P}^{cpt} (1) \to \mathbb{R}_{>0}
\end{align}
there exists an $\eta$ for which $\mathcal{P}_\eta$ is isotopic to $\mathcal{P}$ and 
\begin{align}
   \left| (f_\eta (q_{\bar{e}}) - \beta (q_{\bar{e}})) - (f_\eta (p_{\bar{e}}) - \beta (p_{\bar{e}})) \right| = \ell (e)
\end{align}
for all $e \in \mathcal{P}^{cpt} (1)$ with $\bar{e} \in \mathcal{P}_\eta^{cpt} (1)$ isotopic to $e$.
\end{lemma}
The idea behind this lemma is that, for any given combinatorial type of planar tree, one can prescribe the length of all finite edges. 
\begin{proof}
We start by defining, for any $\tau \in \mathbb{R}^A$ and $e \in \mathcal{P}_\tau^{cpt}$ the function
\[
F_{\tau}(e) = \left| (f_\tau (q_e) - \beta (q_e)) - (f_\tau (p_e) - \beta (p_e)) \right|.
\]
Note that the condition that $\beta$ does not lie on the affine hull of a diagonal ensures that $F_\tau (e) > 0$ for every $e \in \mathcal{P}_\tau^{cpt} (1)$. This follows from a straightforward argument akin to the one given in Proposition~\ref{prop:rays}.

Suppose that $\mathcal{P}$ is realized by $\xi$ so that $\mathcal{P} = \mathcal{P}_\xi$. For a given compact edge $e^\prime \in \mathcal{P}^{cpt} (1)$ and a positive real number $t$ we show that there is an $\eta$ with isotopic tropical $A$-complex for which 
\begin{align} \label{eq:singleedge}
  F_\eta (\bar{e}) = \begin{cases} ( 1+ t) F_\xi (e) & \text{ if } \bar{e} \textnormal{ is isotopic to } e^\prime, \\ F_\xi (e) & \text{ otherwise}.   \end{cases}
\end{align}
Scaling $\xi$ by a small $\lambda$ will contract all the lengths of compact edges to a small value so that the difference on the left is arbitrarily small for all $\bar{e}$.  After making such a sufficiently small contraction, and applying equation~\ref{eq:singleedge} for each edge, we obtain the result. So it suffices to show equation~\ref{eq:singleedge}.

Let $\mathcal{S} = \{ (Q_i, A_i)\}$ be the subdivision dual to $\mathcal{P}$, $g_\xi : Q \to \mathbb{R}$ the piecewise affine function on $Q$ defining $\mathcal{S}$, and $f$ the dual diagonal to $e^\prime$. Then $f$ is a face of polygons $Q_j$ and $Q_k$ in $\mathcal{S}$ which are dual to $p_{e^\prime}$ and $q_{e^\prime}$ respectively. 

By Proposition~\ref{prop:duality}, $p_{e^\prime} = g_{k, \xi}$ and $q_{e^\prime} = g_{l, \xi}$ where $g_{k,\xi} - c_{k,\xi}$ and $g_{l,\xi} - c_{l,\xi}$ are the affine functions on $Q_k$ and $Q_l$ equal to $g_\xi$. Also by this proposition, we have that $f_\xi (p_{e^\prime}) = c_{k, \xi}$ and $f_\xi (q_{e^\prime}) = c_{l, \xi}$.

For $t > 0$, consider the function $h_t: Q \to \mathbb{R}$ defined as
\[
h_t(v) = \min \left\{ 0, t\left[(g_{l,\xi}(v) - c_{l, \xi}) - (g_{k,\xi} (v) - c_{k, \xi} ) \right]\right\}.
\]
Since $f$ is a diagonal of $Q$, it divides $Q$ into two polytopes $Q^\prime$ and $Q^{\prime \prime}$ where $Q_k \subset Q^\prime$ and $Q_l \subset Q^{\prime \prime}$. One observes that 
\begin{align} \label{eq:ht}
    h_t (v) = \begin{cases} 0 & \textnormal{ if } v \in Q^\prime, \\ t\left[(g_{l,\xi}(v) - c_{l, \xi}) - (g_{k,\xi} (v) - c_{k, \xi} ) \right] & \textnormal{ if } v \in Q^{\prime \prime}. \end{cases}
\end{align}
It follows that restricting $h_t$ to $A$ induces this subdivision and the cone $C_{h_t}$ is a face of $C_{\xi}$ by Proposition~\ref{prop:secondaryfan}.  In particular, this subdivision is a coarse subdivision which is refined by $\mathcal{S}_\xi$. 

Define
\[
\eta = \xi + h_t
\]
and observe that $\eta$ induces the same subdivision since $h_t$ lies in a face of the cone $C_{\xi}$. Now, for any compact edge $e \in \mathcal{P}^{cpt} (1)$ not equal to $e^\prime$, with boundary points $p_e, q_e$, we have that their dual polygons $Q_i = \psi^{A} (p_e)$ and $Q_j = \psi^{A} (q_e)$ are either both in $Q^\prime$ or $Q^{\prime \prime}$. Recall that $Q_{\eta}$ is the polyhedron formed by taking the convex hulls of rays $(a, r)$ where $r \leq - \eta (a)$. The compact facets $F_i$ and $F_j$ projecting to $Q_i$ and $Q_j$ respectively and, over $Q_i$ and $Q_j$ are graphs of $g_{i,\xi} - c_{i, \xi} + h_t$ and $g_{j, \xi} - c_{j,\xi} + h_t$. But for such $e$, $h_t$ will be the same affine function for $i$ and $j$ in this case so that
\begin{align*}
    g_{i, \eta} - g_{j, \eta} & = g_{i, \xi} - g_{j, \xi},\\
    c_{i, \eta} - c_{j, \eta} & = c_{i, \xi} - c_{j, \xi}.
\end{align*}
Thus, writing $\bar{e}$ for the edge $\iota_{\mathcal{P}, \mathcal{P}_\eta} (e)$ in $\mathcal{P}_\eta$ isotopic to $e$, it again follows from Proposition~\ref{prop:duality} that 
\[
F_\eta (\bar{e}) = F_\xi (e).
\]

However, for $e^\prime$, the dual polygons $Q_k$ and $Q_l$ are in $Q^\prime$ and $Q^{\prime \prime}$ respectively. This implies
\begin{align*}
    g_{l, \eta} - g_{k, \eta} & = g_{l, \xi}  + t (g_{l,\xi}(v)  - g_{k,\xi} (v) ) - g_{k, \xi}  = (1 + t) (g_{l, \eta} - g_{k, \eta} ) ,\\
     c_{l, \eta} - c_{k, \eta} & = c_{l, \xi}  + t (c_{l,\xi}(v)  - c_{k,\xi} (v) ) - c_{k, \xi}  = (1 + t) (c_{l, \eta} - c_{k, \eta} ).
\end{align*}
Thus 
\[
F_\eta (\bar{e}^\prime) = (1 + t) F_\xi (e^\prime )
\]
implying the result.
\end{proof}
We utilize this result to see that all painted trees can be realized as painted tropical $A$-complexes. To do this, we first recall the description of painted trees from \cite{forcey} based on \cite{boavog}. In this reference, Forcey refers to the vertices of the tree as nodes, allows for bivalent nodes, and paints some of the edges with particular rules on when and how the painting changes. We will additionally direct the edges of any rooted planar tree from the root to the leaves and consider the root and leaves as half edges (i.e. we do not consider there to be root or leaf nodes).
\begin{definition} \label{def:paintedtree}
A painted tree is a planar rooted tree, with each edge either painted or unpainted, satisfying the following conditions
\begin{enumerate}
    \item the root is painted,
    \item the leaves are unpainted,
    \item for a bivalent node, the incoming edge must be painted and the outgoing unpainted,
    \item at any node which is not bivalent, the painting of adjacent edges can be 
    \begin{enumerate}[label=(\alph*)]
        \item  all unpainted,
        \item  all painted,
        \item incoming painted and all outgoing unpainted.
    \end{enumerate}
\end{enumerate}
Given painted trees $t$ and $t^\prime$, we say that $t \prec t^\prime$ if $t^\prime$ can be obtained from $t$ through a sequence of edge contractions.
\end{definition}

One of the main results from \cite{forcey} is the realization of a polytope $\mathcal{J} (m)$, called the multiplihedron, whose face lattice is isomorphic to the poset of painted trees with $m$ leaves. We now show that such multiplihedra can be realized as secondary polytopes using painted tropical $A$-complexes. To do so, we choose an $\alpha \in \mathbb{R}^2$ near the edge of $Q$ from $a_m$ to $a_0$, but outside of $Q$. This is illustrated in Figure~\ref{fig:alpha} for $Q$ a square.
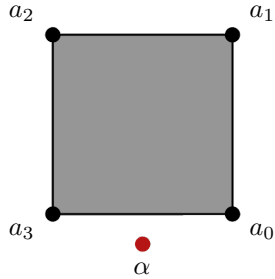
\begin{figure}[ht]
\begin{tikzpicture}[cross line/.style={preaction={draw=white, -, line width=6pt}}]
    \node at (2.6,-1.3) (node A) {$a_0$};
    \node at (-.6,-1.3) (node B) {$a_3$};
    \node at (2.6,1.6) (node C) {$a_1$};
    \node at (-.6,1.6) (node D) {$a_2$};
    \node at (1,-1.8) (node E) {$\alpha$};
    \node[inner sep=0] (image) at (1,0) {\includegraphics[scale=.2]{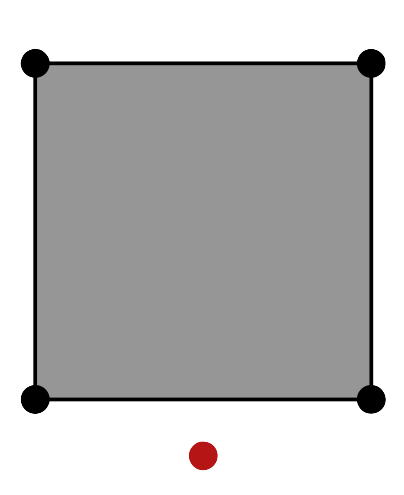}};
\end{tikzpicture}
\caption{\label{fig:alpha} Choice of $\alpha$ outside of edge of $Q$}
\end{figure}
To be more precise in the describing the choice of $\alpha$, let $\nu_i$ be the linear support functions for $Q$ as in equation~\ref{eq:support} for which 
\begin{align*}
    \nu_i (a_{i - 1}) = d_i =  \nu_i (a_i)
\end{align*}
where we index $a_j$ mod $(m + 1)$. Then we assume that $\alpha$ satisfies
\begin{align*}
    \nu_0 (\alpha ) & < d_0, \\
    \nu_j (\alpha ) & > d_j, \text{ for all }j \ne 0. 
\end{align*}
Thus the sign vector defined in equation~\ref{eq:signvec} is 
\[
\mathbf{s}_{A, \alpha} = (+ , - , -, \cdots , - ).
\]
\begin{theorem}
The $m$-th multiplihedron is combinatorially equivalent to the secondary polytope $\Sigma (\bar{A}_\alpha)$. 
\end{theorem}
\begin{proof}
As was mentioned above, for any $\xi \in \mathbb{R}^A$ we direct  the edges of the planar tree $\mathcal{P}_\eta$ from the root to the leaves. For any painted tropical $A$-complex $\mathcal{P}$, we can define a unique painted tree $t_\mathcal{P}$ in the following way. Paint each red edge and leave every blue edge unpainted. For each purple edge $e$, introduce a bivalent node and paint the incoming edge while leaving outgoing edge unpainted. Let us show that this prescription does indeed produce a painted tree by verifying properties (1) - (4) of Definition~\ref{def:paintedtree}.

Property (3) is automatically satisfied by the construction. For property (1) we use Proposition~\ref{prop:rays} and the sign vector of $\alpha$ to see that $f_\eta  - \alpha^* \gg 0$ outside of a sufficiently large bounded set. This implies it is colored either red or purple. Thus the root edge of the painted tree will always be painted. Likewise, Proposition~\ref{prop:rays} asserts that $f_\eta - \alpha^* \ll 0$ outside of a sufficiently large bounded set, which implies it is colored either blue or purple. Thus the leaves of the painted tree will always be unpainted confirming property (2). Finally, if a vertex of $\mathcal{P}$ is red or blue, then the adjacent edges of the associated painted tree are all painted or unpainted respectively. If the vertex is purple, then using the fact $f_\eta$ is concave and is decreasing along leaves and the root, it follows that $f_\eta - \alpha^*$ is strictly decreasing along every edge, thus all outgoing edges have $f_\eta - \alpha^* - c < 0$ along their relative interiors while the incoming edge has $f_\eta - \alpha^* - c > 0$ (and equality at the vertex). But this implies the associated painted tree has all outgoing edges unpainted and the incoming edge painted. Thus property (4) holds as well.

It is immediate that if $\mathcal{P} \prec \mathcal{P}^\prime$ then $t_{\mathcal{P}} \prec t_{\mathcal{P}^\prime}$ as the partial order of painted $A$-complexes arises as either contraction or color change by changing $c$ so that $f_\eta (u) = u(\alpha) + c$ at a vertex $u$. This latter operation corresponds to contracting one or more edges with a bivalent vertex. Thus we have a map of the face lattice of the secondary polytope $\Sigma (\bar{A}_\alpha)$ to that of $\mathcal{J}(m)$. As we can clearly invert the operation to obtain a color map $\kappa$ on a tropical $A$-complex $\mathcal{P}$, it is immediate that this map is injective. However, it is not immediate that the color map $\kappa$ so obtained is always of the form $\kappa_{\eta, c}$ so we must show surjectivity. 

The fact that any painted tree is of the form $t_{\mathcal{P}}$ for a painted tropical $A$-complex $(\mathcal{P}, \kappa)$ follows from Lemma~\ref{lemma:distance} and the earlier observation that $f_\eta - \alpha^*$ is decreasing along all edges of $\mathcal{P}$. Let $t$ be a painted tree and,  suppose   $(\mathcal{P}, \kappa)$ is a tropical $A$-complex which has the same combinatorial type after forgetting bivalent vertices, and whose color function produces $t$. We fix $u$ to be the vertex of the root ray. If the root ray is colored purple (resp. $u$ is purple), we may choose $\xi$ with $\mathcal{P}_\xi$ isotopic to $\mathcal{P}$ and take $c > f_\xi (u) - u(\alpha)$ (resp. $c = f_\xi (u) - u(\alpha)$) to produce $\kappa = \kappa_{\xi, c}$. Thus we will assume that the root ray and vertex are both red.

\begin{figure}[ht]
\begin{tikzpicture}[cross line/.style={preaction={draw=white, -, line width=6pt}}]
    \node[inner sep=0] (image) at (0,0) {\includegraphics[scale=.25]{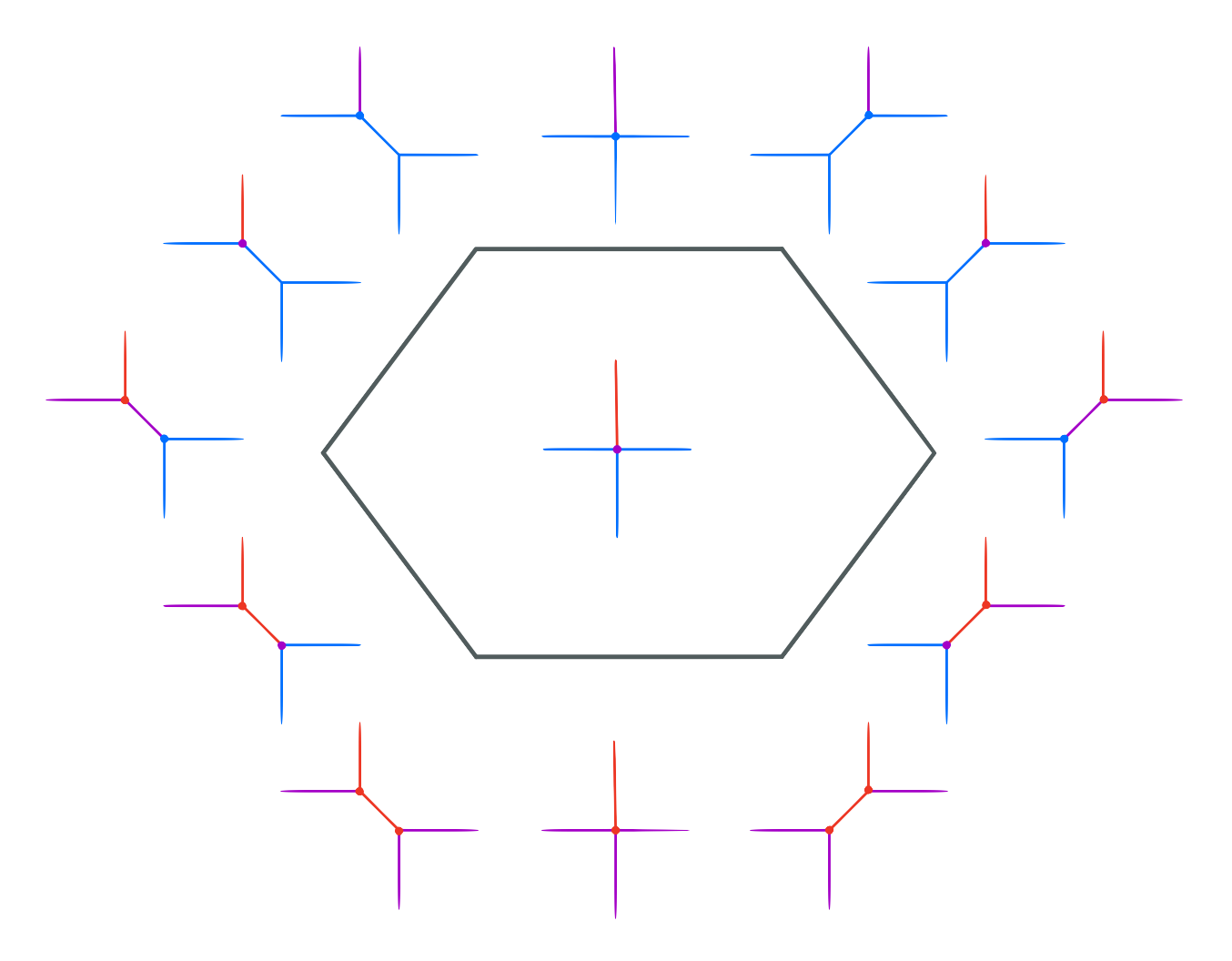}};
\end{tikzpicture}
\caption{\label{fig:paintedtrees}The multiplihedron $\mathcal{J}(3)$ with painted trees and their corresponding faces.}
\end{figure}

Now, recall that we label the boundary points of a  compact edge $e$ as $p_e, q_e$ directed away from the root. Define $\delta (e)$ to be the number of edges to make the unique path from $p_e$ to the vertex $u$ of the root ray and consider the function
$\ell : \mathcal{P}^{cpt} (1) \to \mathbb{R}_{> 0}$ to be 
\begin{align*}
    \ell (e) & = \begin{cases} \frac{1}{m} & \textnormal{ if } \kappa (q_e) = r, \\ 1 - \frac{\delta (e)}{m} & \textnormal{ if } \kappa (q_e) = p, \\ 2 & \textnormal{ if } \kappa (q_e) = b. \end{cases}
\end{align*}
Lemma~\ref{lemma:distance} then asserts the existence of $\eta$ for which 
\[
\left| (f_\eta (q_{\bar{e}}) - \alpha (q_{\bar{e}})) - (f_\eta (p_{\bar{e}}) - \alpha (p_{\bar{e}})) \right| = \ell (e)
\]
Because $f_\eta - \alpha^*$ is decreasing along edges, we can eliminate the absolute value to obtain
\[
 (f_\eta (p_{\bar{e}}) - \alpha (p_{\bar{e}})) - (f_\eta (q_{\bar{e}}) - \alpha (q_{\bar{e}})) = \ell (e).
\]
Let $\bar{u}$ be the root vertex of $\mathcal{P}_\eta$ isotopic to $u$ in $\mathcal{P}$. For a vertex $v$ of $\mathcal{P}_\eta$ other than $\bar{u}$ we may take the (non-empty) unique path of edges $e_1, \ldots, e_k$ starting at $\bar{u}$ to see that 
\[
(f_\eta (\bar{u}) - \alpha (\bar{u})) - (f_\eta (v) - \alpha (v)) = \sum_{i = 1}^k \ell (e_i).
\]
If $\kappa (v) = r$, then all of the edges in the sequence must be red and since there are less than $m$ compact edges we have 
\[
(f_\eta (\bar{u}) - \alpha (\bar{u})) - (f_\eta (v) - \alpha (v)) < 1 \textnormal{ if } \kappa (v) = r.
\]
If $\kappa (v) = p$, then every edge except $e_k$ has $\ell (e) = 1/m$ and $\ell (e_k) = 1 - (k - 1)/ m$ so that 
\[
(f_\eta (\bar{u}) - \alpha (\bar{u})) - (f_\eta (v) - \alpha (v)) = 1 \textnormal{ if } \kappa (v) = p.
\]
Finally, the last case that occurs is
\[
(f_\eta (\bar{u}) - \alpha (\bar{u})) - (f_\eta (v) - \alpha (v)) > 1 \textnormal{ if } \kappa (v) = b.
\]
Thus, taking 
\[
c = f_\eta (\bar{u}) - \alpha (\bar{u}) - 1
\]
one verifies that $\kappa = \kappa_{\eta, c}$.
\end{proof}
For $Q$ and $\alpha$ as in Figure~\ref{fig:alpha}, we illustrate in Figure~\ref{fig:paintedtrees} the painting polytope of $Q$ by $\alpha$ which simultaneously describes the well known painted trees of the multiplihedron $\mathcal{J}(3)$.

\bibliography{painted}{}
\bibliographystyle{plain}
\end{document}